\documentclass[11pt,a4paper,onesided]{article}

\usepackage{ amssymb, amsmath, amsthm, graphics, xspace, enumerate}
\usepackage{graphicx}
\usepackage[a4paper,colorlinks=true,citecolor=blue,urlcolor=blue,linkcolor=blue,bookmarksopen=true,unicode=true,pdffitwindow=true]{hyperref}
 \hypersetup{pdfauthor={Nikola Sandri\'{c}}}
\hypersetup{pdftitle={A transience condition for a class of
one-dimensional symmetric L\'evy processes}}

\newtheorem{tm}{tm}[section]
\newtheorem{theorem}[tm]{Theorem}

\newtheorem{corollary}[tm]{Corollary}
\newtheorem{proposition}[tm]{Proposition}

\newtheorem{definition}[tm]{Definition}

\newcommand{\process}[1]{\{#1_t\}_{t\geq0}}
\newcommand {\R} {\ensuremath{\mathbb{R}}}
\newcommand {\ZZ} {\ensuremath{\mathbb{Z}}}
\newcommand {\N} {\ensuremath{\mathbb{N}}}

\numberwithin{equation}{section}
\def\be{\begin{equation}}
\def\ee{\end{equation}}
\begin{document}

 \title{A transience condition for a class of one-dimensional
symmetric\\ L\'evy processes}
 \author{Nikola Sandri\'{c}\\ Department of Mathematics\\
         Faculty of Civil Engineering, University of Zagreb, Zagreb,
         Croatia \\
        Email: nsandric@grad.hr }

 \maketitle
\begin{center}
{
\medskip

} \end{center}

\begin{abstract}
In this paper, we give a sufficient condition for transience for a
class of one-dimensional symmetric L\'evy processes. More precisely,
we prove that   a one-dimensional symmetric L\'evy process with the
L\'evy measure $\nu(dy)=f(y)dy$ or $\nu(\{n\})=p_n$,  where the
density function $f(y)$ is such that  $f(y)>0$  a.e. and the
sequence $\{p_n\}_{n\geq1}$ is such that $p_n>0$ for all $n\geq1$,
is transient if
$$\int_1^{\infty}\frac{dy}{y^{3}f(y)}<\infty\quad\textrm{or}\quad \sum_{n=1}^{\infty}\frac{1}{n^{3}p_n}<\infty.$$  Similarly, we derive an analogous
transience condition for one-dimensional symmetric random walks with
continuous and discrete jumps.

\end{abstract}


%
%
%
%


\section{Introduction}
Let $(\Omega,\mathcal{F},\mathbb{P})$ be a probability space and let
$\{L_t\}_{t\geq0}$ be a stochastic process on
$(\Omega,\mathcal{F},\mathbb{P})$ taking values in $\R^{d}$,
$d\geq1$. The process $\{L_t\}_{t\geq0}$ is called a \emph{L\'evy
process} if $L_0=0$ $\mathbb{P}$-a.s.,  if it has stationary and
independent increments and if it has c\`adl\`ag paths
$\mathbb{P}$-a.s. (that is, if its trajectories are right-continuous
with left limits $\mathbb{P}$-a.s.). Having these properties,  every
L\'{e}vy process can be completely and uniquely characterized
through the characteristic function of a single random variable
$L_t$, $t>0$, that is, by the famous \emph{L\'evy-Khintchine
formula} we have
$$\mathbb{E}[\exp\{i\langle\xi,L_t\rangle\}]=\exp\{-t\psi(\xi)\},\quad  t\geq0,$$
where
$$\psi(\xi)=i\langle\xi,b\rangle+\frac{1}{2}\langle\xi,c\xi\rangle+\int_{\R^{d}}\left(1-\exp\{i\langle\xi,y\rangle\}+i\langle\xi,y\rangle1_{\{|y|\leq1\}}(y)\right)\nu(dy).$$
Here $b$ is a vector in $\R^{d}$, $c$ is a symmetric
nonnegative-definite $d\times d$ matrix and $\nu(dy)$ is a
$\sigma$-finite Borel measure on $\R^{d}$ satisfying
$$\nu(\{0\})=0\quad\textrm{and}\quad\int_{\R^{d}}\min\{1,y^{2}\}\nu(dy)<\infty.$$
The measure $\nu(dy)$, the triplet $(b, c, \nu)$ and the function
$\psi(\xi)$ are called the \emph{L\'evy measure}, the \emph{L\'evy
triplet} and the \emph{characteristic exponent} of the L\'evy
process $\{L_t\}_{t\geq0}$, respectively. Further, recall that the
vector $b$, the matrix $c$ and the L\'evy measure $\nu(dy)$
correspond  to the deterministic part (shift), the continuous
(Brownian) part and the jumping part of the L\'evy process
$\{L_t\}_{t\geq0}$, respectively.

In this paper, we consider  the transience and recurrence property
of L\'{e}vy processes. A L\'evy process $\{L_t\}_{t\geq0}$ is said
to be \emph{transient} if
$$\mathbb{P}\left(\lim_{t\longrightarrow\infty}|L_t|=\infty\right)=1,$$
and \emph{recurrent} if
$$\mathbb{P}\left(\liminf_{t\longrightarrow\infty}|L_t|=0\right)=1.$$
It is well known that every L\'{e}vy process is either transient or
recurrent (see \cite[Theorem 35.3]{sato-book}). An equivalent
definition (characterization) of the transience and recurrence
property of L\'{e}vy processes can be given through the sojourn
times. A L\'evy process $\{L_t\}_{t\geq0}$ is transient if and only
if
$$\mathbb{E}\left[\int_0^{\infty}1_{\{|L_t|<a\}}(L_t)dt\right]<\infty\quad\textrm{for all}\ a>0.$$
Similarly,  a L\'evy process $\{L_t\}_{t\geq0}$ is recurrent if and
only if
$$\mathbb{E}\left[\int_0^{\infty}1_{\{|L_t|<a\}}(L_t)dt\right]=\infty\quad\textrm{for all}\ a>0$$
 (see \cite[Theorem
35.4]{sato-book}). The above characterizations of the transience and
recurrence property   are not applicable in most cases. A more
operable
 characterization, by using the nice analytical characterization of L\'{e}vy processes through the L\'evy-Khintchine formula, has been given by the
 well-known
\emph{Chung-Fuchs criterion}. A L\'evy process $\{L_t\}_{t\geq0}$ is
transient if and only if
$$\int_{\{|\xi|<a\}}\rm{Re}\left(\frac{1}{\psi(\xi)}\right)\it{d}\xi<\infty
\quad \textrm{for some}\ a>\rm{0}$$ (see \cite[Corollary 37.6 and
Remark 37.7]{sato-book}). Again, in many situations this criterion
is also not applicable. More precisely, for a given L\'evy triplet
$(b,c,\nu)$ it is not always easy to compute the integral part of
the characteristic exponent as well as the integral appearing in the
Chung-Fuchs criterion. According to this, the aim of this paper is
to derive a sufficient condition for transience for L\'evy processes
in terms of the L\'evy triplet. Let us remark that analogous
definitions and characterizations of the transience and recurrence
property hold also for random walks (see \cite[Chapter
4]{durrett-book}). Recall that a \emph{random walk} is a stochastic
process $\{S_n\}_{n\geq0}$ defined on a probability space
$(\Omega,\mathcal{F},\mathbb{P})$ taking values in $\R^{d}$,
$d\geq1$, defined by $S_0:=0$ and $S_n:=\sum_{i=1}^{n}J_i$, where
 $\{J_n\}_{n\geq1}$ is a sequence of i.i.d. random variables called the \emph{jumps} of $\{S_n\}_{n\geq0}$.

As already mentioned, in this paper we consider  the one-dimensional
symmetric case only. Note that, according to \cite[Theorem
37.8]{sato-book} and \cite[Theorem 4.2.13]{durrett-book}, the
limitation to the one-dimensional case is not   too big restriction
since it is well known that every $d$-dimensional, $d\geq3$,
L\'{e}vy process and random walk are transient. Further, recall that
a stochastic process $\{X_t\}_{t\in\mathbb{T}}$ is \emph{symmetric}
if $\{X_t\}_{t\in\mathbb{T}}\stackrel{\hbox{\scriptsize
d}}{=}\{-X_t\}_{t\in\mathbb{T}}$, where $\mathbb{T}=[0,\infty)$ or
$\{0,1,2,\ldots\}$ and $\stackrel{\hbox{\scriptsize d}}{=}$ means
that the processes $\{X_t\}_{t\in\mathbb{T}}$ and
$\{-X_t\}_{t\in\mathbb{T}}$
 have the same
finite-dimensional distributions. In the L\'evy process case,
$\{L_t\}_{t\geq0}$ is symmetric if and only if $b=0$ and $\nu(dy)$
is a symmetric measure, that is, $\nu(B)=\nu(-B)$ holds for all
Borel sets $B\subseteq\R^{d}$ (see \cite[Exercise 18.1]{sato-book}),
while a random walk  is symmetric if and only if its jumps have a
symmetric distribution.
 Now, let us state
the main results of this paper.
\begin{theorem}\label{tm1.1} Let $\{L_t\}_{t\geq0}$ be a one-dimensional symmetric L\'evy process
with the L\'evy measure  $\nu(dy)=f(y)dy$ or $\nu(\{n\})=p_n$, where
the  density function $f(y)$ is such that  $f(y)>0$  a.e. and the
sequence $\{p_n\}_{n\geq1}$ is such that $p_n>0$ for all $n\geq1$.
Then, $\{L_t\}_{t\geq0}$ is transient if
\begin{align}\label{eq:1.1}\int_1^{\infty}\frac{dy}{y^{3}f(y)}<\infty\quad\textrm{or}\quad \sum_{n=1}^{\infty}\frac{1}{n^{3}p_n}<\infty.\end{align}
\end{theorem}
By using \cite[Theorem 38.2]{sato-book}, the above transience
condition can be generalized to the general symmetric case.
\begin{theorem}\label{tm1.2} Let $\{L^{1}_t\}_{t\geq0}$ and $\{L^{2}_t\}_{t\geq0}$ be  one-dimensional symmetric L\'evy
processes with the L\'evy measures  $\nu_1(dy)$ and $\nu_2(dy)$.
Further, let  $\nu_1(dy)$ be  as in Theorem \ref{tm1.1} and let it
satisfy the condition  (\ref{eq:1.1}). If
$$\int_0^{\infty}y^{2}|\nu_1-\nu_2|(dy)<\infty,$$  then the transience property of $\{L^{1}_t\}_{t\geq0}$ implies
the transience property of $\{L^{2}_t\}_{t\geq0}$. Here $|\cdot|$
denotes the total variation norm on the space of signed
measures.\end{theorem}

Let us remark that the same transience condition  holds also in the
case of  one-dimensional symmetric random walks.  More precisely,
let $\{S_n\}_{n\geq0}$ be a one-dimensional  symmetric random walk
with jumps $\mathbb{P}(J_1\in dy)=f(y)dy$ or
$\mathbb{P}(J_1=n)=p_n$, where $f(y)>0$ a.e. and $p_n>0$ for all
$n\geq1$, then the condition (\ref{eq:1.1}) implies the transience
property of $\{S_n\}_{n\geq0}$. Also, let us  remark that, according
to \cite[Theorem 38.2]{sato-book} or \cite[Lemma 1.2]{shepp1}, the
assumptions $f(y)>0$ a.e. and $p_n>0$ for all $n\geq1$ can be
relaxed. More precisely, it suffices to demand positivity of $f(y)$
and $p_n$ on the complement of a compact set.

A simple application of the condition (\ref{eq:1.1}) is in the class
of stable processes. Recall that a one-dimensional symmetric stable
L\'evy process or a random walk is given by the characteristic
exponent $\psi(\xi)=\gamma|\xi|^{\alpha}$ or by the L\'evy triplet
$(0,c,\nu),$ where $\gamma>0$, $\alpha\in(0,2]$,
$$c=\left\{\begin{array}{cc}
                                                      0, & \alpha<2 \\
                                                      2\gamma, &
                                                      \alpha=2
                                                    \end{array}\right.\quad\textrm{and}\quad \nu(dy)=\left\{\begin{array}{cc}
                                                      \gamma\frac{\alpha
                 2^{\alpha-1}\Gamma(\frac{\alpha+1}{2})}{\pi^{\frac{1}{2}}\Gamma(1-\frac{\alpha}{2})}|y|^{-\alpha-1}dy , & \alpha<2 \\
                                                      0, &
                                                      \alpha=2.
                                                    \end{array}\right.$$ It is well known, as a consequence of the Chung-Fuchs criterion, that this process is transient if and only if $\alpha<1.$
Further, recall that a probability density function of a
one-dimensional symmetric stable distribution behaves like
$c_{\alpha}|y|^{-\alpha-1}$ when $|y|\longrightarrow\infty$, for
$\alpha\in(0,2)$ and $$c_\alpha=\left\{\begin{array}{cc}
                                                      \frac{\gamma}{2}, & \alpha=1 \\
                                                      \frac{\gamma}{\pi}\Gamma(\alpha+1)\sin\left(\frac{\pi\alpha}{2}\right),&
                                                      \alpha\neq1
                                                    \end{array}\right.$$(see \cite[Remark 14.18]{sato-book}). Now, as a
simple consequence of Theorem \ref{tm1.1}, we get a new proof for
the transience property of one-dimensional symmetric stable L\'evy
processes and random walks.
\begin{corollary}
A one-dimensional symmetric stable L\'evy process or a random walk
is transient if $\alpha<1.$
\end{corollary}
Note that the above corollary implies that the function
$y\longmapsto y^{3}$, appearing in the condition (\ref{eq:1.1}), is
optimal in the class of power functions.

The transience and recurrence property of one-dimensional symmetric
L\'evy processes  in terms of the L\'evy triplet has already been
studied in the literature. Namely, in \cite[Theorem 38.3]{sato-book}
(see also \cite[Theorem 5]{shepp1})  it has been proved that  a
one-dimensional symmetric L\'evy process $\{L_t\}_{t\geq0}$ with the
L\'evy measure $\nu(dy)$ is recurrent if
\begin{align}\label{eq:1.2}\int_1^{\infty}\left(\int_0^{y}z\nu\left(\max\{1,z\},\infty\right)dz\right)^{-1}dy=\infty.\end{align}
Intuitively, the condition (\ref{eq:1.2}) measures the speed of
divergence of the second moment of  $\nu(dy)$, and, regarding  this
speed, it concludes the recurrence property. Clearly, if $\nu(dy)$
has finite second moment, then $\{L_t\}_{t\geq0}$ is recurrent.
Thus, if the second moment of $\nu(dy)$ diverges slow enough, then
$\{L_t\}_{t\geq0}$ is recurrent. Similarly, the condition
(\ref{eq:1.1}) measures the speed of divergence of the third moment
of $\nu(dy)$. If the third moment of  $\nu(dy)$ diverges fast
enough, then $\{L_t\}_{t\geq0}$ is transient.

Recall that a symmetric Borel measure $\rho(dy)$ on $\R$ is
\emph{unimodal} if it is finite outside of any neighborhood of the
origin and if $x\longmapsto\rho(x,\infty)$ is a convex function on
$(0,\infty)$. Equivalently, a symmetric Borel measure $\rho(dy)$ on
$\R$ is unimodal if it is of the form
$\rho(dy)=a\delta_0(dy)+f(y)dy,$ where $0\leq a\leq\infty$ and the
density function $f(y)$ is  symmetric, decreasing on $(0,\infty)$
and it satisfies $\int_{|y|>\varepsilon}f(y)dy<\infty$ for all
$\varepsilon>0$ (see \cite[Chapter 5]{sato-book}). Note that
measures with a discrete support are never unimodal. Now,  if the
L\'evy measure $\nu(dy)$ is additionally unimodal, the condition
(\ref{eq:1.2}) is also necessary for the recurrence property (see
\cite[Theorem 38.3]{sato-book} or \cite[Theorem 5]{shepp1}). Also,
note that unimodality of the  $\nu(dy)$ and finiteness of
(\ref{eq:1.2}) imply that $f(y)>0$ a.e., and the condition
(\ref{eq:1.2}) is stronger than the condition (\ref{eq:1.1}) (see
Section 4 for the proof). Thus, the condition (\ref{eq:1.1}) is a
generalization of the condition (\ref{eq:1.2}) in the case when the
jumping measure is not unimodal.  For the necessity of unimodality
for the characterization of the transience property  by the
condition (\ref{eq:1.2}) see \cite[Theorems 38.2, 38.3 and 38.4 and
Lemma 38.8]{sato-book} or \cite[Theorems 4 and 5]{shepp1} and
\cite[Theorem 1]{shepp2}.

Finally, we give an example where the
 condition (\ref{eq:1.1}) is more suitable than the Chung-Fuchs criterion and the condition (\ref{eq:1.2}).
 We consider an example of a L\'evy process with ``multiple indices of
 stability".
Let $\{L_t\}_{t\geq0}$ be a one-dimensional symmetric L\'evy process
with the L\'evy measure $\nu(\{n\})=p_n$,  where
$p_{2n}=(2n)^{-\alpha-1}$ and $p_{2n-1}=(2n-1)^{-\beta-1}$ for
$n\geq1$ and $\alpha,\beta\in(0,\infty)$. For a continuous version
of such  process it suffices to interpolate the points
$\{(i,p_i):i\in\ZZ\}$. Now, clearly, if $\alpha<1$ and $\beta<1$,
then the condition (\ref{eq:1.1}) implies transience of
$\{L_t\}_{t\geq0}$. On the other hand, since $\nu(dy)$ is not
unimodal, the condition (\ref{eq:1.2}) is not applicable, and the
application of the Chung-Fuchs criterion leads to a non-trivial
computation. Also, let us  remark that the same example shows that
the condition (\ref{eq:1.1}) is only sufficient for transience.
Indeed, assume that $\alpha<1$ and $\beta\geq1$. Then, since
$\beta\geq1$, the condition (\ref{eq:1.1}) fails to hold. On the
other hand, since $\alpha<1$, $\{L_t\}_{t\geq0}$ is transient (see
Section 4 for the proof).

Now, we explain our strategy of  proving Theorem \ref{tm1.1}. The
proof is divided in three steps. In the first step, by using
electrical networks techniques, we prove Theorem \ref{tm1.1} in the
case of a random walk with discrete jumps. In the second step, we
prove Theorem \ref{tm1.1} in the case of a random walk
$\{S_n\}_{n\geq0}$ with continuous jumps $\mathbb{P}(J_1\in
dy)=f(y)dy$. More precisely, for $\delta>0$ we define a
discretization of $\{S_n\}_{n\geq0}$ as a random walk
$\{S^{\delta}_n\}_{n\geq0}$ on $\delta\ZZ$ with the jump
distribution $\mathbb{P}(J_1^{\delta}=\delta n):=\int_{\delta
n-\frac{\delta}{2}}^{\delta n+\frac{\delta}{2}}f(y)dy,$ $n\in\ZZ.$
Next, by an ``approximation approach" we prove that all the random
walks $\{S^{\delta}_n\}_{n\geq0}$, $\delta>0,$  are either transient
or recurrent at the same time and their transience and recurrence
property is equivalent with the transience and recurrence property
of $\{S_n\}_{n\geq0}$. Finally, by using the first step, we prove
that the condition (\ref{eq:1.1}) for $\{S_n\}_{n\geq0}$ implies the
transience property of $\{S^{1}_n\}_{n\geq0}$. And this accomplishes
the proof of the second step. At the end, in the last step, we
consider the case of L\'evy processes. By using \cite[Theorem
38.2]{sato-book}, it suffices to consider the situation of a
compound Poisson process. Now, the proof  follows from the first and
second step. This accomplishes the proof of Theorem \ref{tm1.1}.

The paper is organized as follows.  In Section 2, we give a proof of
Theorem \ref{tm1.1} for the case of discrete jumps. In Section 3, by
using the results from Section 2, we  proof  Theorem \ref{tm1.1} for
the case of continuous jumps. Finally, in Section 4, we discuss some
properties of the condition (\ref{eq:1.1}).

\section{Discrete case}
In this section, we prove the main step of the proof of Theorem
\ref{tm1.1}. More precisely,  we derive a sufficient condition for
transience for one-dimensional symmetric random walks on $\ZZ$.
\begin{theorem}\label{tm2.1}
Let $\{S_n\}_{n\geq0}$ be a one-dimensional symmetric random walk on
$\ZZ$ with jumps $\mathbb{P}(J_1=n)=p_n$, where the sequence
$\{p_n\}_{n\geq1}$ is such that $p_n>0$ for all $n\geq1$. Then the
random walk $\{S_n\}_{n\geq0}$ is transient if
\begin{align}\label{eq:2.1}\displaystyle\sum_{n\geq
1}\frac{1}{n^{3}p_n}<\infty.\end{align}
\end{theorem}
Note that the same transience condition also holds in the case of a
one-dimensional symmetric L\'evy process $\{L_t\}_{t\geq0}$ with a
discrete supported L\'evy measure $\nu(\{n\})=p_n$, where $p_n>0$
for all $n\geq1.$ Indeed, first note that
$$\{L_t\}_{t\geq0}\stackrel{\hbox{\scriptsize
d}}{=}\{S_{P_t}\}_{t\geq0},$$ where $\{S_n\}_{n\geq0}$ is a random
walk with jumps $\mathbb{P}(J_1=n):=\frac{1}{\nu(\ZZ)}p_n$ and
$\{P_t\}_{t\geq0}$ is the Poisson process with parameter $\nu(\ZZ)$
independent of $\{S_n\}_{n\geq0}$. Now, the desired result follows
from the definition of transience in terms of sojourn times.

The proof of Theorem \ref{tm2.1} is based on  techniques and results
from electrical networks. Let us introduce some notation we need. A
\emph{graph} is a pair $G=(V(G), E(G))$ where $V(G)$ is a set of
\emph{vertices} and $E(G)$ is a symmetric subset of $V(G)\times
V(G)$, called the \emph{edge set}. By symmetry we mean that
$(u,v)\in E(G)$ if and only if $(v,u)\in E(G)$.  For two vertices
$u,v\in V(G)$ such that $(u,v)\in E(G)$, we say that $u$ and $v$ are
\emph{adjacent} and write $u\sim v$ and by $e_{uv}$ we denote the
edge which connects them. A \emph{path} in a graph is a sequence of
vertices where each successive pair of vertices is an edge in the
graph. A graph is \emph{connected} if there is a path from any of
its vertices to any other. A \emph{network} is a pair $N=(G,c)$,
where $G$ is a connected graph and $c$ is a function
$c:E(G)\longrightarrow[0,\infty)$ called \emph{conductance}. In the
sequel we assume that a network $N$ satisfies
$$c(u):=\displaystyle\sum_{v\sim u}c(e_{uv})<\infty$$ for all
$u\in V(G).$ A \emph{random walk} on a  network $N$ is a
time-homogeneous Markov chain $\{X_n\}_{n\geq0}$ with the state
space $V(G)$ and transition kernel
$$q_{uv}:=\mathbb{P}(X_1=v|X_0=u)=\left\{\begin{array}{cc}
                                                      \frac{c(e_{uv})}{c(u)}, & u\sim v \\
                                                      0, &
                                                      \textrm{otherwise}.
                                                    \end{array}\right.$$
Note that the Markov chain $\{X_n\}_{n\geq0}$ is irreducible (that
is, $\sum_{n=1}^{\infty}\mathbb{P}(X_n=v|X_0=u)>0$ for all $u,v\in
V(G)$) and it is reversible (that is, there exists a nontrivial
measure $\pi(dy)$ on $V(G)$, such that $\pi(u)q_{uv}=\pi(v)q_{vu}$
for all $u,v\in V(G)$). Indeed, irreducibility easily follows from
connectedness of the graph $G$ and for the reversibility measure we
can take $\pi:=c.$ Also, let us  remark that to every irreducible
and reversible time-homogeneous Markov chain on a discrete state
space $S$ given by the transition kernel $q_{uv}$, $u,v\in S$, and a
reversibility measure $\pi(dy)$ we can join a network $N=(G,c)$.
Indeed, put $V(G)=S$, the vertices $u$ and $v$ are adjacent if
$q_{uv}>0$, the graph $G$ is connected because of irreducibility of
the corresponding Markov chain and the conductance is defined by
$c(e_{uv})=\pi(u)q_{uv}$.

Further, let  $u_0\in V(G)$ be an arbitrary vertex of the network
$N$.
 A \emph{flow} from $u_0$ to $\infty$ is a  function
$\theta:V(G)\times V(G)\longrightarrow\R$ such that $\theta(u,v)=0$
unless $u\sim v$, $\theta(u,v)=-\theta(v,u)$ for all $u,v\in V(G)$
and $\displaystyle\sum_{v\in V(G)}\theta(u,v)=0$ if $u\neq u_0$. We
call the flow  a \emph{unit flow} if $\displaystyle\sum_{u\in
V(G)}\theta(u_0,u)=1$. The \emph{energy} of the flow is defined by
$$\mathcal{E}(\theta)=\frac{1}{2}\displaystyle\sum_{u\sim
v}\frac{\theta^{2}(u,v)}{c(e_{uv})}.$$ Next, recall that a state $u$
of a time-homogeneous Markov chain $\{X_n\}_{n\geq0}$ on a discrete
state space $S$ is called \emph{transient} if
$\sum_{n=1}^{\infty}\mathbb{P}(X_n=u|X_0=u)<\infty$ and it is called
\emph{recurrent} if
$\sum_{n=1}^{\infty}\mathbb{P}(X_n=u|X_0=u)=\infty.$ If every state
is transient (resp. recurrent) the chain itself is called transient
(resp. recurrent). It is well known that every irreducible Markov
chain is either recurrent or transient (see \cite[Theorem
8.1.2]{meyn-tweedie-book}). Finally, the main tool for proving
Theorem \ref{tm2.1} is given in the following theorem.

\begin{theorem}\label{tm2.2}\cite[Theorem 1]{lyons}
Random walk on a network $N$ is transient if and only if there is a
unit flow on $N$ of finite energy from some  vertex to $\infty.$
\end{theorem}

\begin{proof}[Proof of Theorem \ref{tm2.1}]
First, note that, according to \cite[Lemma 1.2]{shepp1}, without
loss of generality we can assume that $p_0>0$. Thus,
$\{S_n\}_{n\geq0}$ is a random walk on the network $N=(G,c)$, where
$G=(V(G),E(G))=(\ZZ,\ZZ\times\ZZ)$ and $c(e_{uv})=p_{|v-u|}.$ Now,
following the ideas from \cite[Theorem 1]{hobert}, we construct a
unit flow from $0$ to $\infty$ for the random walk
$\{S_n\}_{n\geq0}$ such that the corresponding energy is bounded
from above by (\ref{eq:2.1}). Then the desired result follows from
Theorem \ref{tm2.2}.  First, let us partition the set of vertices
$V(G)=\ZZ$ on the sets $B_0=\{0\}$,
$B_i=\{2^{i-1},2^{i-1}+1,\ldots,2^{i}-1\}$ and
$B_{-i}=\{-2^{i}+1,\ldots,-2^{i-1}-1,-2^{i-1}\},$  $i\geq1$, and let
us define a unit flow $\theta:V(G)\times V(G)\longrightarrow\R$ from
$0$ to $\infty$  in the following way. For $u\in B_i$ and $v\in B_j$
define
$$\theta(u,v):=\left\{\begin{array}{cc} \frac{1}{2}, & i=0\ \textrm{and}\ j=-1 \ \textrm{or}\ j=1\\
                                                      0, &  i=j \ \textrm{or} \ |i-j|\geq2  \\
2^{-2|i|}, & 0<i<j=i+1\ \textrm{or}\
                                                      j<j+1=i<0.
                                                    \end{array}\right.$$
Recall that  flow has to be antisymmetric, hence we define
 $\theta(v,u):=-\theta(u,v)$. Next, note that
$$\displaystyle\sum_{v\in\ZZ}\theta(1,v)=\theta(1,0)+\theta(1,2)+\theta(1,3)=-\frac{1}{2}+\frac{1}{4}+\frac{1}{4}=0,$$
and analogously $\displaystyle\sum_{v\in\ZZ}\theta(-1,v)=0.$
Further, for $u\in B_i$,  $i\geq 2$, we have
$$\displaystyle\sum_{v\in\ZZ}\theta(u,v)=\displaystyle\sum_{v\in B_{i-1}}\theta(u,v)+\displaystyle\sum_{v\in
B_{i+1}}\theta(u,v)=-2^{i-2}2^{-2(i-1)}+2^{i}2^{-2i}=0,$$ and
analogously  for  $u\in B_i$, $i\leq-2$, we have
$$\displaystyle\sum_{v\in\ZZ}\theta(u,v)=0.$$
According to this, $\theta$ is a flow from $0$ to $\infty$. Finally,
since
$\displaystyle\sum_{v\in\ZZ}\theta(0,v)=\theta(0,1)+\theta(0,-1)=1$,
$\theta$ is a unit flow from  $0$ to $\infty$. Now, let us prove
that the energy of the flow $\theta$ is bounded from above by
(\ref{eq:2.1}). We have
\begin{align*}\mathcal{E}(\theta)&=\frac{1}{2}\displaystyle\sum_{u\sim
v}\frac{\theta^{2}(u,v)}{c(e_{uv})}=\frac{1}{2}\displaystyle\sum_{(u,v)\in
E(G)}\frac{\theta^{2}(u,v)}{p_{|v-u|}}\\&=
\frac{1}{2}\displaystyle\sum_{u\geq 0,\, v\geq 0,\, u\neq
v}\frac{\theta^{2}(u,v)}{p_{|v-u|}}+\frac{1}{2}\displaystyle\sum_{u\geq
0, \, v<0}\frac{\theta^{2}(u,v)}{p_{|v-u|}}\\&\ \ \ \
+\frac{1}{2}\displaystyle\sum_{u<0, \,
v\geq0}\frac{\theta^{2}(u,v)}{p_{|v-u|}}+\frac{1}{2}\displaystyle\sum_{u<0,\,
v<0,\, u\neq v}\frac{\theta^{2}(u,v)}{p_{|v-u|}}.\end{align*} Note
that, from the symmetry of the distribution  $\{p_n\}_{n\in\ZZ}$ and
the definition of the flow  $\theta$, the second and the third therm
equal $\frac{1}{8p_1}$. Next,  again from the symmetry of the
distribution $\{p_n\}_{n\in\ZZ}$ and the symmetry  of the function
$\theta^{2}$, we have
\begin{align*}\mathcal{E}(\theta)&\leq\frac{1}{4p_1}+\displaystyle\sum_{u\geq
0, \, v\geq
0}\frac{\theta^{2}(u,v)}{p_{|v-u|}}\\&=\frac{1}{4p_1}+\displaystyle\sum_{u\geq
0,\, w\geq0}\frac{\theta^{2}(u,u+w)}{p_w}+\displaystyle\sum_{u\leq0,
\, w\leq
0}\frac{\theta^{2}(u,u+w)}{p_{|w|}}\\&=\frac{1}{4p_1}+2\displaystyle\sum_{u\geq
0, \, w\geq0}\frac{\theta^{2}(u,u+w)}{p_w}.\end{align*} Note that
$\theta(u,u+w)=0$ when $u+w\geq 4u$, except for $u=0$ and $w=1$.
Thus
\begin{align*}\mathcal{E}(\theta)&\leq\frac{3}{4p_1}+2\displaystyle\sum_{w\geq
2, \,
u\geq\lceil\frac{w}{3}\rceil}\frac{\theta^{2}(u,u+w)}{p_w}\\
&=\frac{3}{4p_1}+\frac{1}{8p_2}+2\displaystyle\sum_{w\geq 2, \,
u\geq\lceil\frac{w}{3}\rceil, \,
u\neq1}\frac{\theta^{2}(u,u+w)}{p_w},\end{align*} where $\lceil
x\rceil$ denotes the smallest integer not less than $x$. Now, since
for $u\in B_i$, $i\geq2$, (that is, for $u\geq2$), we have
$\theta^{2}(u,v)\leq(2^{-2(i-1)})^{2}=16(2^{i})^{-4}\leq16u^{-4}$,
then \begin{align*}\displaystyle\sum_{u\geq\lceil\frac{w}{3}\rceil,
\, u\neq1}^{\infty}\theta^{2}(u,u+w)\leq\left\{\begin{array}{cc}
                                                      \displaystyle\int_{\lceil\frac{w}{3}\rceil-1}^{\infty}\frac{16}{x^{4}}dx=\frac{16}{3(\lceil\frac{w}{3}\rceil-1)^{3}}\leq\frac{16}{3(\frac{w}{3}-1)^{3}}=\frac{144}{(w-3)^{3}}, & w\geq4\\
                                                      \displaystyle\int_{1}^{\infty}\frac{16}{x^{4}}dx=\frac{16}{3}, &
                                                      w=2,3.
                                                    \end{array}\right.\end{align*}
This yields
\begin{align*}\mathcal{E}(\theta)&\leq\frac{3}{4p_1}+\frac{1}{8p_2}+2\sum_{w\geq2}\frac{1}{p_w}\displaystyle\sum_{
u\geq\lceil\frac{w}{3}\rceil, \,
u\neq1}\theta^{2}(u,u+w)\\
&\leq\frac{3}{4p_1}+\frac{1}{8p_2}+\frac{32}{3p_2}+\frac{32}{3p_3}+288\displaystyle\sum_{w\geq4}\frac{1}{(w-3)^{3}p_w}.\end{align*}
This accomplishes the proof of Theorem \ref{tm2.1}.
\end{proof}

\section{Continuous case}

In this section, we prove Theorem \ref{tm1.1} in the case of
continuous jumps. As in the case of discrete jumps,  the main step
is to consider the random walk case.
\begin{theorem}\label{tm3.1}Let $\{S_n\}_{n\geq0}$ be one-dimensional symmetric random walk with jumps $\mathbb{P}(J_1\in
dy)=f(y)dy$, where  the probability density function $f(y)$ is such
that $f(y)>0$ a.e. Then the random walk $\{S_n\}_{n\geq0}$ is
transient if
\begin{align}\label{eq:3.1}\int_{1}^{\infty}\frac{dy}{y^{3}f(y)}<\infty.\end{align}
\end{theorem}
Again, similarly as in the case of discrete jumps, the transience
condition for the L\'evy process case can be easily derived from the
random walk case.
  Indeed, let $\{L_t\}_{t\geq0}$ be a one-dimensional symmetric L\'evy process with the L\'evy measure
$\nu(dy)=f(y)dy$, where the density $f(y)$ is such that $f(y)>0$
a.e. Then, first note that, according to \cite[Theorem
38.2]{sato-book}, without loss of generality we can assume that
$\nu(\R)<\infty.$ Thus,
$$\{L_t\}_{t\geq0}\stackrel{\hbox{\scriptsize
d}}{=}\{S_{P_t}\}_{t\geq0},$$ where $\{S_n\}_{n\geq0}$ is a random
walk with continuous jumps $\mathbb{P}(J_1\in
dy):=\frac{1}{\nu(\R)}f(y)dy$ and $\{P_t\}_{t\geq0}$ is the Poisson
process with parameter $\nu(\R)$ independent of $\{S_n\}_{n\geq0}$.
Now, the desired result follows from the definition of transience in
terms of sojourn times.

Before the proof of Theorem \ref{tm3.1}, we need some auxiliary
results. Let $B\subseteq\R^{d}$ be an arbitrary Borel set and let us
denote by $\mathbb{D}(\R^{d})$  the space of $\R^{d}$-valued
c\`adl\`ag functions equipped with the Skorohod topology. Define the
\emph{set of recurrent paths} by
$$R(B):=\{\omega\in\ \mathbb{D}(\R^{d}):\forall n\in\N,\ \exists t\geq n \ \textrm{such that} \ \omega(t)\in B\},$$ and the \emph{set of transient
paths}  by
$$T(B):=\{\omega\in\ \mathbb{D}(\R^{d}):\exists s\geq0 \ \textrm{such that} \ \omega(t)\not\in B, \ \forall t\geq s
\}.$$  In the following proposition, we characterize the transience
and recurrence property of L\'evy process in terms of c\`adl\`ag
paths.

\begin{proposition}\label{p3.2} Let $\textbf{L}=\{L_t\}_{t\geq0}$ be an $\R^{d}$-valued L\'evy
process. Then, $\textbf{L}$ is transient if and only if
$\mathbb{P}_{\textbf{L}}(T(B_a))=1$ for all $a>0$, and it is
recurrent if and only if $\mathbb{P}_{\textbf{L}}(R(B_a))=1$ for all
$a>0$, where $B_a$ denotes the open ball of radius $a$ around the
origin.
\end{proposition}
\begin{proof} The proof follows directly from the definition of the
transience and recurrence properties.
\end{proof}
Now, let us  recall the notion of characteristics of a
semimartingale (see \cite{jacod-book}).  Let
$(\Omega,\mathcal{F},\process{\mathcal{F}},$
$\mathbb{P},\process{S})$, $\process{S}$ in the sequel, be a
one-dimensional semimartingale and let $h:\R\longrightarrow\R$ be a
truncation function (that is, a continuous bounded function such
that $h(x)=x$ in a neighborhood of the origin).
 We  define two processes
$$\check{S}(h)_t:=\sum_{s\leq t}(\Delta S_s-h(\Delta S_s))\quad
\textrm{and} \quad S(h)_t:=S_t-\check{S}(h)_t,$$ where the process
$\process{\Delta S}$ is defined by $\Delta S_t:=S_t-S_{t-}$ and
$\Delta S_0:=S_0$. The process $\process{S(h)}$ is a special
semimartingale. Hence, it admits the unique decomposition
\begin{align}\label{eq:3.2}S(h)_t=S_0+M(h)_t+B(h)_t,\end{align} where $\process{S(h)}$ is a local
martingale and $\process{S(h)}$ is a predictable process of bounded
variation.

\begin{definition}
Let $\process{S}$  be a semimartingale and let
$h:\R\longrightarrow\R$ be the truncation function. Furthermore, let
$\process{B(h)}$  be the predictable process
 of bounded variation appearing in (\ref{eq:3.2}),  let $N(\omega,ds,dy)$ be the
compensator of the jump measure
$$\mu(\omega,ds,dy)=\sum_{s:\Delta S_s(\omega)\neq 0}\delta_{(s,\Delta S_s(\omega))}(ds,dy)$$ of the process
$\process{S}$ and let $\process{C}$ be the quadratic co-variation
process for $\process{S^{c}}$ (continuous martingale part of
$\process{S}$), that is,
$$C_t=\langle S^{c}_t,S^{c}_t\rangle.$$  Then $(B,C,N)$ is called
the \emph{characteristics} of the semimartingale $\process{S}$
 (relative to $h(x)$). If we put $\tilde{C}(h)_t:=\langle
 M(h)_t,M(h)_t\rangle$, where $\process{M(h)}$ is the local martingale
 appearing in (\ref{eq:3.2}), then $(B,\tilde{C},N)$ is called the \emph{modified
 characteristics} of the semimartingale $\process{S}$ (relative to $h(x)$).
\end{definition}

 \begin{proposition}\label{p3.4} Let   $\textbf{S}=\{S_n\}_{n\geq0}$ be a one-dimensional random
 walk with continuous jumps $\mathbb{P}(J_1\in dy)=f(y)dy$. For $\delta>0$, let
 $\textbf{S}^{\delta}=\{S^{\delta}_n\}_{n\geq0}$ be a random walk on $\delta\ZZ$ with
 discrete jumps
$$\mathbb{P}(J^{\delta}_1=\delta n)=\int_{\delta n-\frac{\delta}{2}}^{\delta
n+\frac{\delta}{2}}f(y)dy,\quad n\in\ZZ.$$ Further, let
$\process{P}$ be the Poisson process with parameter $1$ independent
of $\textbf{S}$ and $\textbf{S}^{\delta}$, $\delta>0$, and let
$\bar{\textbf{S}}:=\{S_{P_t}\}_{t\geq0}$ and
$\bar{\textbf{S}}^{\delta}:=\{S^{\delta}_{P_t}\}_{t\geq0}$. Then
$\bar{\textbf{S}}^{\delta}\stackrel{\hbox{\scriptsize
d}}{\longrightarrow} \bar{\textbf{S}}\ \ \textrm{when}\ \
\delta\longrightarrow 0,$ where $\stackrel{\hbox{\scriptsize
d}}{\longrightarrow}$ denotes the convergence in
$\mathbb{D}(\R^{d})$ with respect to the Skorohod topology, and
 all  the
random walks $\textbf{S}^{\delta}$, $\delta>0$, are either transient
 or  recurrent at the same time and this transience and recurrence dichotomy is equivalent with the transience and recurrence
 dichotomy
of the random walk $\textbf{S}$.

\end{proposition}
\begin{proof}
 Clearly,
 $\bar{\textbf{S}}$ and $\bar{\textbf{S}}^{\delta}$,
$\delta>0$, are processes of bounded variation. Thus, they are
semimartingales. Further, let $h(x)$ be the truncation function and
let $(B,\tilde{C},N)$ and
$(B^{\delta},\tilde{C}^{\delta},N^{\delta})$, $\delta>0$, be the
modified characteristics of $\bar{\textbf{S}}$ and
$\bar{\textbf{S}}^{\delta}$, $\delta>0$, respectively. Now, since
$\bar{\textbf{S}}$ and $\bar{\textbf{S}}^{\delta}$ are L\'evy
processes,  by \cite[Proposition 2.17 and Corollary
II.4.19]{jacod-book}, their (modified) characteristics are exactly
the corresponding L\'evy triplets, that is,
$$B_t=t\mathbb{E}[h(J_1)], \quad
N(ds,dy)=ds\mathbb{P}(J_1\in dy),\quad
\tilde{C}_t=t\mathbb{E}[h^{2}(J_1)]$$ and
$$B^{\delta}_t=t\mathbb{E}[h(J^{\delta}_1)], \quad
N^{\delta}(ds,dy)=ds\mathbb{P}(J^{\delta}_1\in dy),\quad
\tilde{C}^{\delta}_t=t\mathbb{E}[h^{2}(J^{\delta}_1)].$$ According
to this, in order to prove the desired convergence, by \cite[Theorem
VIII.2.17]{jacod-book}, it suffices to show that
$$\sup_{s\leq t}|B^{\delta}_s-B_s|\longrightarrow0,\quad
\tilde{C}^{\delta}_t\longrightarrow
\tilde{C}_t\quad\textrm{and}\quad
\int_{[0,t]\times\R}g(y)N^{\delta}(ds,dy)\longrightarrow
\int_{[0,t]\times\R}g(y)N(ds,dy)$$ when $\delta\longrightarrow0$ for
all $t\geq0$ and for every bounded and continuous function $g(x)$
vanishing in a neighborhood of the origin.  Clearly, in order to
prove the above convergences, it suffices to show that
$$\mathbb{E}[g(J^{\delta}_1)]\longrightarrow \mathbb{E}[g(J_1)]$$
when $\delta\longrightarrow0$ for every bounded and continuous
function $g(x)$. But this fact easily follows from \cite[Theorem
2.1]{billingsley}, definition of the jumps
$\{J_n^{\delta}\}_{n\geq0}$, $\delta>0$, and  continuity of the
jumps $\{J_n\}_{n\geq0}$.

Now, we prove the second part of the proposition. Let $\delta_0>0$
be arbitrary. By
  completely the same arguments as above, we have $\bar{\textbf{S}}^{\delta}\stackrel{\hbox{\scriptsize
d}}{\longrightarrow} \bar{\textbf{S}}^{\delta_0}\ \ \textrm{when}\ \
\delta\longrightarrow \delta_0.$ Next, let $a>0$ be arbitrary, then
 $T(B_a)=R(B_a)^{c}$,
$R(B_a)$ is open  in $\mathbb{D}(\R)$ and $\partial R(B_a)\subseteq
R(B_{a+\varepsilon})\setminus R(B_{a-\varepsilon})$, where $\partial
R(B_a)$ denotes the boundary of the set $R(B_a)$ and
$0<\varepsilon<a$. Thus, by Proposition \ref{p3.2}, we have
$$\mathbb{P}_{\bar{\textbf{S}}^{\delta_0}}(\partial
R(B_a))\leq\mathbb{P}_{\bar{\textbf{S}}^{\delta_0}}(R(B_{a+\varepsilon}))-\mathbb{P}_{\bar{\textbf{S}}^{\delta_0}}(R(B_{a-\varepsilon}))=0$$for
all $a>0.$ Hence, the sets  $T(B_a)$ and $R(B_a)$, $a>0$, are
continuity sets for $\mathbb{P}_{\bar{\textbf{S}}^{\delta_0}}$. Now,
by \cite[Theorem 2.1]{billingsley}, this yields
             $$\lim_{\delta\longrightarrow \delta_0}\mathbb{P}_{\bar{\textbf{S}}^{\delta}}(T(B_a))=\mathbb{P}_{\bar{\textbf{S}}^{\delta_0}}(T(B_a))$$ for all $a>0.$
Hence, for all $a>0$, the function $$\delta\longmapsto
             \mathbb{P}_{\bar{\textbf{S}}^{\delta}}(T(B_a))$$ is continuous on $(0,\infty)$.
             According to this, by   Proposition \ref{p3.2}, $\mathbb{P}_{\bar{\textbf{S}}^{\delta}}(T(B_a))=1$
            for all
             $\delta>0$ and all $a>0$, or $\mathbb{P}_{\bar{\textbf{S}}^{\delta}}(T(B_a))=0$
            for all
             $\delta>0$ and all $a>0$. This means, again by  Proposition \ref{p3.2},
             that all the random walks $\textbf{S}^{\delta}$, $\delta>0$, are either
             transient
             or recurrent at the same time.

            Finally, by completely the same arguments as above, $\mathbb{P}_{\bar{\textbf{S}}}(\partial T(B_a))=0$ for all
            $a>0$. Then, again by   \cite[Theorem 2.1]{billingsley}, we have
             \begin{align*}\lim_{\delta\longrightarrow
             0}\mathbb{P}_{\bar{\textbf{S}}^{\delta}}(T(B_a))=\mathbb{P}_{\bar{\textbf{S}}}(T(B_a))\end{align*}
             for all $a>0$.
             Thus, by Proposition \ref{p3.2}, the transience and recurrence property of the random walks $\textbf{S}^\delta$, $\delta>0$, is equivalent with the transience and recurrence
 property
of the random walk $\textbf{S}$.
\end{proof}

At the end, we prove Theorem \ref{tm3.1}.
\begin{proof}[Proof of Theorem \ref{tm3.1}]
Let $\{S^{1}_n\}_{n\geq0}$ be a random walk on $\ZZ$ with discrete
 jumps
$$\mathbb{P}(J^{1}_1=n):=\int_{n-\frac{1}{2}}^{n+\frac{1}{2}}f(y)dy,\quad
n\in\ZZ.$$ Next, by the Jensen's inequality, we have
\begin{align*}\infty>\int_{\frac{1}{2}}^{\infty}\frac{dy}{y^{3}f(y)}&=\sum_{n=1}^{\infty}\int_{n-\frac{1}{2}}^{n+\frac{1}{2}}\frac{dy}{y^{3}f(y)}\geq
\sum_{n=1}^{\infty}\frac{1}{\left(n+\frac{1}{2}\right)^{3}\int_{n-\frac{1}{2}}^{n+\frac{1}{2}}f(y)dy}=\sum_{n=1}^{\infty}\frac{1}{\left(n+\frac{1}{2}\right)^{3}\mathbb{P}(J^{1}_1=n)}.\end{align*}
Thus, by Theorem \ref{tm2.1}, the random walk $\{S^{1}_n\}_{n\geq0}$
is transient. Now, the desired result follows from  Proposition
\ref{p3.4}.
\end{proof}

\section{Some remarks on the main results}
In this section, we discus some properties of the condition
(\ref{eq:1.1}) we mentioned in Section 1. First, we prove that,
under the assumption of unimodality, the condition (\ref{eq:1.2}) is
stronger than the condition (\ref{eq:1.1}). Recall that a
one-dimensional symmetric L\'evy measure $\nu(dy)$ is unimodal if it
is of the form $\nu(dy)=f(y)dy$, where the density function $f(y)$
is  symmetric, decreasing on $(0,\infty)$ and it satisfies
$\int_{|y|>\varepsilon}f(y)dy<\infty$ for all $\varepsilon>0$.
 Let us fix $y_0>1$, then,
by the Fubini's theorem, for all $y\geq y_0$ we have
\begin{align*}\int_0^{y}z\nu\left(\max\{1,z\},\infty\right)dz&=\int_0^{y}z\int_{\max\{1,z\}}^{\infty}f(u)dudz\\&=\int_0^{1}z\int_{1}^{\infty}f(u)dudz+\int_{1}^{y}z\int_{z}^{\infty}f(u)dudz\\
&=\frac{1}{2}\int_{1}^{\infty}f(u)du+\int_{1}^{y}z\int_{z}^{\infty}f(u)dudz\\
&=\frac{1}{2}\int_1^{y}u^{2}f(u)du+\frac{y^{2}}{2}\int_y^{\infty}f(u)du\\
&\geq\frac{y^{3}-1}{6}f(y)+\frac{y^{2}}{2}\int_y^{\infty}f(u)du\\
&\geq Cy^{3}f(y),
\end{align*}
where in the fifth line we used the fact that $f(y)$ is decreasing
on $(0,\infty)$ and  $0<C<\frac{y_0^{3}-1}{6y_0^{3}}$ is arbitrary.
Now, we have
\begin{align*}&\int_1^{\infty}\left(\int_0^{y}z\nu\left(\max\{1,z\},\infty\right)dz\right)^{-1}dy\\&=\int_1^{y_0}\left(\int_0^{y}z\nu\left(\max\{1,z\},\infty\right)dz\right)^{-1}dy+\int_{y_0}^{\infty}\left(\int_0^{y}z\nu\left(\max\{1,z\},\infty\right)dz\right)^{-1}dy\\
&\leq\frac{2(y_0-1)}{y_0\int_{y_0}^{\infty}f(y)dy}+\frac{1}{C}\int_{y_0}^{\infty}\frac{dy}{y^{3}f(y)}\\
&\leq D\int_1^{\infty}\frac{dy}{y^{3}f(y)},\end{align*} for some
suitably chosen constant $D>0$. Therefore, we have proved the
desired result.

Finally, we prove the transience property of a L\'evy process with
``multiple indices of stability" $\{L_t\}_{t\geq0}$. Recall that
 $\{L_t\}_{t\geq0}$ is a one-dimensional symmetric L\'evy process
with the L\'evy measure $\nu(\{n\})=p_n$,  where
$p_{2n}=(2n)^{-\alpha-1}$ and $p_{2n-1}=(2n-1)^{-\beta-1}$ for
$n\geq1$ and $\alpha,\beta\in(0,\infty)$. We claim that if
$\alpha<1$ and $\beta\geq1$, then $\{L_t\}_{t\geq0}$ is transient.
Clearly, it suffices to consider the random walk case. Let
$\{S_n\}_{n\geq0}$ be a random walk on $\ZZ$ with jumps
$\mathbb{P}(J_1=n)=c^{-1}p_n,$ where $c:=\sum_{n\in\ZZ}p_n$ is the
norming constant and $p_n$, $n\geq1$, are as above. First, let us
define a sequence of stopping times $\{T_n\}_{n\geq0}$ inductively
by $T_0:=0$ and
$$T_n:=\inf\{k>T_{n-1}:S_k\in 2\ZZ\},$$  for $n\geq1,$ and let us prove that $\mathbb{P}(T_n<\infty)=1$ for all $n\geq1$.
We have
\begin{align*}\mathbb{P}(T_1=\infty)&=\mathbb{P}(S_k\in2\ZZ+1 \
\textrm{for all}\ k\geq1)\\&=
\displaystyle\lim_{k\longrightarrow\infty}\mathbb{P}(S_l\in2\ZZ+1 \
\textrm{for all}\ 1\leq l\leq k)\\
&=\mathbb{P}(J_1\in2\ZZ+1)\displaystyle\lim_{k\longrightarrow\infty}(\mathbb{P}(J_1\in2\ZZ))^{k-1}=0.\end{align*}
Now, let us assume that $\mathbb{P}(T_{n-1}<\infty)=1$ and prove
that  $\mathbb{P}(T_n<\infty) = 1,$ $n\geq2$. Denote by
$N:=T_{n-1}$. Then, by the strong Markov property, we have
\begin{align*}\mathbb{P}(T_n<\infty) &=
\mathbb{E}[1_{\{T_n<\infty\}}] =
\mathbb{E}[1_{\{T_1<\infty\}}\circ\theta_N] =
\mathbb{E}[\mathbb{E}[1_{\{T_{1}<\infty\}}\circ\theta_N|\mathfrak{F}_N]]\\
&= \mathbb{E}[\mathbb{E}^{S_N} [1_{\{T_1<\infty\}}]]
=\displaystyle\sum_{i\in\ZZ}\mathbb{E}[1_{\{S_N=2i\}}] =
1,\end{align*} where $\theta_n$, $n\geq0$, are the shift operators
on the canonical state space $\ZZ^{\{0,1,2,\ldots\}}$ defined by
$(\theta_n\omega)(m):=\omega(n+m)$, $m\geq0$, and
$\mathcal{F}_N:=\{A\in\mathcal{F}:A\cap
\{N=n\}\in\sigma\{S_1,\ldots,S_n\} \ \textrm{for all}\ n\geq1\}.$
Thus, the Markov chain $X_n:=S_{T_n}$ is well defined. Clearly,
$\{X_n\}_{n\geq0}$ is irreducible  on $2\ZZ$. Further, note that
$\{X_n\}_{n\geq0}$ and $\{S_n\}_{n\geq0}$ are transient or recurrent
at the same time. Indeed, let us define the following stoping times
$\tau:=\inf\{n\geq1:S_n=0\}$ and
$\tilde{\tau}=\inf\{n\geq1:X_n=0\}$.  We have
$$\mathbb{P}(\tilde{\tau}=\infty)=\mathbb{P}(X_n\neq0 \ \textrm{for all}\  n\geq1)=\mathbb{P}(S_n\neq0\ \textrm{for all}\ n\geq1)=\mathbb{P}(\tau=\infty).$$
Now, the desired result follows from \cite[Propositions 8.1.3 and
8.1.4]{meyn-tweedie-book}. According to this, it suffices to prove
the transience property of $\{X_n\}_{n\geq0}$. Further, note that
$\{X_n\}_{n\geq0}$ is actually a symmetric random walk on $2\ZZ$.
Indeed, for $n\geq0$ and $i,j\in\ZZ$ we have
\begin{align*}&\mathbb{P}(X_{n+1}=2j|X_{n}=2i)\\
&=\mathbb{P}(S_1=2j|S_0=2i)+\displaystyle\sum_{i_1\in\ZZ}\mathbb{P}(S_1=2i_1+1|S_0=2i)\mathbb{P}(S_2=2j|S_1=2i_1+1)+\ldots\\
&=\mathbb{P}(S_1=2j-2i|S_0=0)+\displaystyle\sum_{i_1\in\ZZ}\mathbb{P}(S_1=2i_1-2i+1|S_0=0)\mathbb{P}(S_2=2j-2i|S_1=2i_1-2i+1)+\ldots\\
&=\mathbb{P}(X_{n+1}=2j-2i|X_{n}=0)\\
&=\mathbb{P}(X_1=2j-2i).
\end{align*} Thus, $\{X_n\}_{n\geq0}$ is spatially homogeneous.
 Next, for $n,k\geq0$ and $i,j\in\ZZ$ we have
\begin{align*} \mathbb{P}(X_{n+k}-X_n=2i)&=\displaystyle\sum_{j\in\ZZ}\mathbb{P}(X_{n+k}=2i+2j, \
X_n=2j)\\&=\displaystyle\sum_{j\in\ZZ}\mathbb{P}(X_{n+k}=2i+2j|X_n=2j)\mathbb{P}(X_n=2j)\\&=\mathbb{P}(X_k=2i),\end{align*}
and for $k\geq1$, $n_1,\ldots,n_k\geq0$, $0\leq n_1\leq\ldots\leq
n_k$, and $i_1,\ldots,i_{k-1}\in\ZZ$ we have
\begin{align*}&\mathbb{P}(X_{n_k}-X_{n_{k-1}}=2i_{k-1},\ldots,
X_{n_2}-X_{n_1}=2i_1)\\&=\displaystyle\sum_{j\in\ZZ}\mathbb{P}(X_{n_k}=2i_{k-1}+\ldots+2i_2+2j,\ldots,X_{n_2}=2i_1+2_j,X_{n_1}=2j)\\
&=\displaystyle\sum_{j\in\ZZ}\mathbb{P}(X_{n_k}=2i_{k-1}+\ldots+2i_2+2j|X_{n_{k-1}}=2i_{k-2}+\ldots+2i_2+2j)\cdots\\
&\ \ \ \ \ \ \ \,\
\mathbb{P}(X_{n_2}=2i_1+2j|X_{n_1}=2j)\mathbb{P}(X_{n_1}=2j)\\
&=\mathbb{P}(X_{n_k-n_{k-1}}=2i_{k-1})\cdots\mathbb{P}(X_{n_2-n_1}=2i_1)\\&=\mathbb{P}(X_{n_k}-X_{n_{k-1}}=2i_{k-1})\cdots\mathbb{P}(X_{n_2}-X_{n_1}=2i_1).\end{align*}
Symmetry is trivially satisfied. Thus, the claim follows. Finally,
let us show that the random walk $\{X_n\}_{n\geq0}$ is transient.
For $i\in\ZZ\setminus\{0\}$ we have
\begin{align*}\mathbb{P}(X_1=2i)&=\mathbb{P}(S_1=2i)+\displaystyle\sum_{j\in\ZZ}\mathbb{P}(S_1=2j+1)\mathbb{P}(S_2=2i|S_1=2j+1)+\ldots\\&\geq
\mathbb{P}(S_1=2i)=c^{-1}p_{2i}=c^{-1}|2i|^{-\alpha-1}.\end{align*}
Now, since $\alpha<1$, from the condition (\ref{eq:1.1}) we have
$$\displaystyle\sum_{n=1}^{\infty}\frac{1}{(2n)^{3}\mathbb{P}(X_1=2n)}<\infty.$$
Therefore, we have proved the desired result.

\bibliographystyle{alpha}
\bibliography{References}

\begin{thebibliography}{Dur10}

\bibitem[Bil99]{billingsley}
P.~Billingsley.
\newblock {\em Convergence of probability measures}.
\newblock John Wiley \& Sons Inc., New York, second edition, 1999.

\bibitem[Dur10]{durrett-book}
R.~Durrett.
\newblock {\em Probability: theory and examples}.
\newblock Cambridge University Press, Cambridge, fourth edition, 2010.

\bibitem[HS02]{hobert}
J.~P. Hobert and J.~Schweinsberg.
\newblock Conditions for recurrence and transience of a {M}arkov chain on
  {$\Bbb Z\sp +$} and estimation of a geometric success probability.
\newblock {\em Ann. Statist.}, 30(4):1214--1223, 2002.

\bibitem[JS03]{jacod-book}
J.~Jacod and A.~N. Shiryaev.
\newblock {\em Limit theorems for stochastic processes}.
\newblock Springer-Verlag, Berlin, second edition, 2003.

\bibitem[Lyo83]{lyons}
T.~Lyons.
\newblock A simple criterion for transience of a reversible {M}arkov chain.
\newblock {\em Ann. Probab.}, 11(2):393--402, 1983.

\bibitem[MT93]{meyn-tweedie-book}
S.~P. Meyn and R.~L. Tweedie.
\newblock {\em Markov chains and stochastic stability}.
\newblock Springer-Verlag London Ltd., London, 1993.

\bibitem[Sat99]{sato-book}
K.~Sato.
\newblock {\em L\'evy processes and infinitely divisible distributions}.
\newblock Cambridge University Press, Cambridge, 1999.

\bibitem[She62]{shepp1}
L.~A. Shepp.
\newblock Symmetric random walk.
\newblock {\em Trans. Amer. Math. Soc.}, 104:144--153, 1962.

\bibitem[She64]{shepp2}
L.~A. Shepp.
\newblock Recurrent random walks with arbitrarily large steps.
\newblock {\em Bull. Amer. Math. Soc.}, 70:540--542, 1964.

\end{thebibliography}

\end{document}